\documentclass[12pt]{amsart}
\title[Hitting times, Bridges, and Schr\"odinger's equation]{On Hitting times, Bessel bridges, and Schr\"odinger's equation}
\author{Gerardo Hern\'andez-del-Valle}
\email{gerardo@stat.columbia.edu}
\address{Statistics Dept., Columbia University, 1255 Amsterdam Ave. Room 1005, New York, N.Y. 10027.}
\date{September 12, 2010}
\subjclass[2000]{Primary: 60J65,45D05,60J60; Secondary: 45G15,
45G10, 45Q05, 45K05.}
\keywords{First boundary problems, Schr\"odinger's equation, Bessel bridges, heat equation}
\newtheorem{theorem}{Theorem}[section]
\newtheorem{corollary}[theorem]{Corollary}
\newtheorem{problem}[theorem]{Problem}
\newtheorem{definition}[theorem]{Definition}
\newtheorem{proposition}[theorem]{Proposition}
\newtheorem{remark}[theorem]{Remark}

\begin{document}
\maketitle
\begin{abstract}
In this paper we establish relationships between four important concepts: (a) hitting time problems of Brownian motion, (b) 3-dimensional Bessel bridges,  (c) Schr\"odinger's equation with linear potential, and (d) heat equation problems with moving boundary. We relate (a) and (b) by means of Girsanov's theorem, which suggests a strategy to extend our ideas to problems in $\mathbb{R}^n$ and general diffusions. This approach also leads to (c) because we may relate, through a Feynman-Kac representation, functionals of a Bessel bridge with two Schr\"odinger-type problems. Finally, the relationship between (c) and (d) suggests a possible link between Generalized Airy processes and their hitting times.
\end{abstract}

\section{Introduction}

Finding the density for the first time a Brownian motion reaches a moving boundary is an old and classic problem with a wide range of applications in both mathematics and applied sciences. The problem may be traced back  to Louis Bachelier's doctoral thesis [Bachelier (1900)], Th\'eorie de la Sp\'eculation,  and to a paper by the Austrian physicist Erwin Schr\"odinger  [Schr\"odinger (1915)].

On the other hand, Robbins and Siegmund (1973) and later Groeneboom (1987), Salminen (1988) and Martin-L\"of (1998) are among the first to explicitly establish the relationship between hitting time problems and parabolic partial differential equations.

The aim of this work is not to provide a new set of solutions to the problem of boundary crossing. Our goal is instead conceptual: we establish relationships between  (a) hitting time problems of Brownian motion, (b) 3-dimensional Bessel bridge,  (c) Schr\"odinger's equation with linear potential, and (d) heat equation problems with moving boundary. We determine conditions under which  these problems are equivalent.

The paper is organized as follows: The hitting time problem is introduced in Section \ref{sec2}, and by use of Girsanov's theorem, we relate it to  a functional of  a 3-dimensional Bessel bridge. This relationship is  established in Section \ref{sec3}.

In Section \ref{sec4} we derive the Feynman-Kac representation that will allow us to determine the correspondence between Bessel bridges and Schr\"odinger's equation with linear potential. Next, in Section \ref{sec5} we derive a pair of equivalent problems which in Section \ref{sec6} are used to link boundary crossing probabilities with the heat equation. We conclude in Section \ref{sec7} with some final remarks. 
 
\section{The problem}\label{sec2}
Throughout this paper $B$ stands for a one-dimensional standard Brownian motion, and $f(\cdot)$ is a real-valued function denoting a ``moving boundary''. Moreover
\begin{eqnarray}\label{stop}
T:=\inf\left\{t\geq 0|B_t=f(t)\right\}
\end{eqnarray}
is the first time that $B$ reaches the moving boundary $f$. The main motivation of this work is to study the density  of $T$. To be more  precise, we consider the following problem.
\begin{problem}\label{pr}
Determine the density of $T$, given that, for $a>0$
\begin{equation}\label{f}
f(t):=a+\int_0^tf'(u)du,\quad f''(t)\geq 0,\enskip\hbox{and}\enskip \int_0^t\left(f'(u)\right)^2du<\infty
\end{equation}
for $t>0$.
\end{problem}

To analize Problem \ref{pr} it is convenient to note that using (\ref{stop}) and setting
\begin{eqnarray}\label{bn}
\tilde{B}_t:=B_t-\int_0^tf'(u)du
\end{eqnarray}
we obtain the following alternative representation of $T$:
\begin{eqnarray}
\nonumber T&=&\inf\left\{t\geq 0|B_t=a+\int_0^tf'(u)du\right\}\\
\nonumber&=&\inf\left\{t\geq 0|B_t-\int_0^tf'(u)du=a\right\}\\
\label{stop1}&=&\inf\left\{t\geq 0|\tilde{B}_t=a\right\}.
\end{eqnarray}
Since our study will make use of Girsanov's theorem [see for instance: Section 3.5, Chapter 3 in Karatzas and Shreve (1991)], let us first introduce the heuristics that link these ideas:
\subsection{Girsanov's theorem and hitting time problems} Notice that if there exists a measure $\tilde{\mathbb{P}}$ under which the process $\tilde{B}$ is a one-dimensional standard Brownian motion, then for $T$ as in (\ref{stop}) [or (\ref{stop1})] and $\tilde{B}$ as in (\ref{bn}), then
\begin{eqnarray*}
\tilde{\mathbb{P}}(T<t)=\int_0^th(s,a)ds,\qquad t>0,
\end{eqnarray*}
where 
\begin{eqnarray}\label{level}
h(s,a):=\frac{|a|}{\sqrt{2\pi s^3}}\exp\left\{-\frac{a^2}{2s}\right\},\quad s>0,\enskip a\not=0.
\end{eqnarray}
This function $h$ is the hitting time density of Brownian motion reaching a fixed level $a$. Alternatively, $h$  is also referred to as the {\it derived\/} source solution of the heat equation [see for instance Rosenbloom and Widder (1959)].

This last statement follows from observing, in  (\ref{stop1}), that the original problem is reduced to that of hitting a `fixed' boundary $a$.

Assuming such  measure $\tilde{\mathbb{P}}$ exists, how is it related to the original measure $\mathbb{P}$? More explicitly, what is the connection (if any) between $\mathbb{P}(T<t)$ and $\tilde{\mathbb{P}}(T<t)$?

To answer this question we first need to introduce, the so-called $\mathbb{P}$ exponential martingales:\\[0.2cm]
\noindent{\bf Remark.} {\it Exponential martingale and Novikov's condition.} Girsanov's theorem states that the process $\tilde{B}$ in  (\ref{bn}) will be a one-dimensional standard Brownian motion under the  measure $\tilde{\mathbb{P}}$, 
as long as the following exponential process
\begin{eqnarray}\label{m}
Z_t:=\exp\left\{\int_0^t f'(u)dB_u-\frac{1}{2}\int_0^t(f'(u))^2 du\right\}\qquad 0\leq t<\infty
\end{eqnarray}
is a $\mathbb{P}$-{\it martingale\/}. In turn, for this last statement to hold, a sufficient condition is {\it Novikov's condition\/} [see for instance, Proposition 3.5.12 in Karatzas and Shreve (1991)] :
\begin{eqnarray}\label{n}
\mathbb{E}\left[\exp\left\{\frac{1}{2}\int_0^t(f'(u))^2du\right\}\right]<\infty\qquad\forall\,  0\leq t<\infty.
\end{eqnarray} 
\indent Now, given the $\mathbb{P}$-martingale $Z$ in  (\ref{m}), together with the $\tilde{\mathbb{P}}$-martingale
\begin{eqnarray}\label{m1}
\tilde{Z}_t:=\exp\left\{-\int_0^t f'(u)d\tilde{B}_u-\frac{1}{2}\int_0^t(f'(u))^2 du\right\}\quad 0\leq t<\infty,
\end{eqnarray}
where $\tilde{B}$ is a  Brownian motion under $\tilde{\mathbb{P}}$,
it follows from Girsanov's theorem that the  measures $\mathbb{P}$ and $\tilde{\mathbb{P}}$ are equivalent and related as follows:\\[0.2cm]
 If (\ref{n}) holds  and $Z$ and $\tilde{Z}$ are as in (\ref{m}) and (\ref{m1}), respectively, then the pairs $(\mathbb{P},B)$ and $(\tilde{\mathbb{P}},\tilde{B})$ satisfy the following identities:
\begin{eqnarray*}
\tilde{\mathbb{P}}\left(\tilde{B}_t\in A\right)&=&\mathbb{E}\left[Z_t\mathbb{I}_{(\tilde{B}_t\in A)}\right]\\
&=&\mathbb{E}\left[Z_t\mathbb{I}_{(B_t-\int_0^tf'(u)du\in A)}\right]
\end{eqnarray*}
and
\begin{eqnarray}\label{second}
\nonumber\mathbb{P}\left(B_t\in A\right)&=&\tilde{\mathbb{E}}\left[\tilde{Z}_t\mathbb{I}_{(B_t\in A)}\right]\\
&=&\tilde{\mathbb{E}}\left[\tilde{Z}_t\mathbb{I}_{(\tilde{B}_t+\int_0^tf'(u)du\in A)}\right].
\end{eqnarray}

This sequence of ideas leads to the following representation:
\begin{proposition}\label{prop1}
Given that $T$ is as in (\ref{stop}) and condition (\ref{f}) holds,  the exponential process $\tilde{Z}$ in (\ref{m1}) is a $\tilde{\mathbb{P}}$-martingale. Furthermore,  given the hitting time density $h$ in  (\ref{level}), it follows
\begin{eqnarray}
\nonumber&&\mathbb{P}\left(T<t\right)\\
\label{pro}&&\enskip=\int_0^te^{-f'(t)a-\frac{1}{2}\int_0^s(f'(u))^2du}\tilde{\mathbb{E}}\left[e^{\int_0^sf''(u)\tilde{B}_udu}\Bigg{|}T=s\right]h(s,a)ds,
\end{eqnarray}
for $t\geq 0$.
\end{proposition}
\begin{proof} The first claim is justified from (\ref{f}) and Novikov's condition, equation (\ref{n}).
Next, from  Girsanov's theorem and the representation (\ref{second}), we have the following equality
\begin{eqnarray*}
\mathbb{P}(T<t)=\tilde{\mathbb{E}}\left[\tilde{Z}_t\mathbb{I}_{(T<t)}\right],
\end{eqnarray*}
where $\tilde{Z}$ is as in (\ref{m1}). 

From the fact that $\tilde{Z}$ is a $\tilde{\mathbb{P}}$-martingale and $\mathbb{I}_{(T<t)}$ is $\mathcal{F}_{T\wedge t}$-measurable, it follows from the Optional Sampling Theorem that
\begin{eqnarray}
\nonumber\mathbb{P}(T<t)&=&\tilde{\mathbb{E}}\left[\tilde{\mathbb{E}}(\tilde{Z}_t|\mathcal{F}_{T\wedge t})\,\mathbb{I}_{(T<t)}\right]\\
\nonumber&=&\tilde{\mathbb{E}}\left[\tilde{Z}_{t\wedge T}\,\mathbb{I}_{(T<t)}\right]\\
\label{ss}&=&\tilde{\mathbb{E}}\left[\tilde{Z}_{T}\,\mathbb{I}_{(T<t)}\right].
\end{eqnarray}
Finally, from (\ref{ss}) together with the integration by parts formula:
\begin{eqnarray*}
e^{-\int_0^t f'(u)d\tilde{B}_u-\frac{1}{2}\int_0^t(f'(u))^2 du}=e^{-\tilde{B}_tf'(t)+\int_0^tf''(u)\tilde{B}_udu-\frac{1}{2}\int_0^t(f'(u))^2 du}.
\end{eqnarray*}
Then the terminal condition $\tilde{B}_T=a$, and conditioning with respect to $T$ under $\tilde{\mathbb{P}}$ completes the proof of the proposition.
\end{proof}
In view of (\ref{pro}), our study now focuses on the following expectation:
\begin{eqnarray*}
\tilde{\mathbb{E}}\left[e^{\int_0^sf''(u)\tilde{B}_udu}\Bigg{|}T=s\right].
\end{eqnarray*}
Thus, our next task is to understand  the dynamics of  the process $\tilde{B}_u$ (for $0\leq u\leq s$), which is conditioned to hit  the fixed level $a$ for the {\it first  time\/} at time $s$.
\section{Three dimensional Bessel bridge}\label{sec3}
The first idea that comes to mind is that the conditioned process $\tilde{B}_u$ (for $0\leq u\leq s$) is a Brownian bridge, since at $t=0$, $\tilde{B}_0=0$ and at time $s$, $\tilde{B}_s=a$. However, this argument is flawed because, in contrast to a Brownian bridge,  the process  $\tilde{B}$ can {\it never\/} be above level $a$ before time $s$.

Elaborating on the heuristics,
let us for the moment write the conditioned process as: $\hat{B}_u$. Alternatively, setting
\begin{eqnarray*}
\tilde{X}_u:=a-\hat{B}_u,\quad 0\leq u\leq s
\end{eqnarray*}
we shall analyze the dynamics of the process $\tilde{X}$:
\begin{eqnarray}\label{sab}
\tilde{X}_0=a;\qquad \tilde{X}_u>0\enskip\hbox{for}\enskip 0<u<s; \qquad\tilde{X}_s=0.
\end{eqnarray}
That is, $\tilde{X}$ is {\it strictly positive\/} (Bessel part), except at time $s$, when it is  {\it pinned\/} down at zero (bridge part). It may be formally proved that the process just described is in fact the following.
\begin{definition}\label{bb3} A 3-dimensional Bessel bridge is a stochastic process which has the following dynamics:
\begin{eqnarray}\label{bessel}
d\tilde{X}_t:=\left[\frac{1}{\tilde{X}}_t-\frac{\tilde{X}_t}{(s-t)}\right]dt+dW_t,\qquad \tilde{X}_0=a\geq 0,
\end{eqnarray}
$0\leq t\leq s$, and $W$ is a one-dimensional standard Brownian motion.
Furthermore, letting
\begin{eqnarray}\label{ph}
k(\sigma,\kappa):=\frac{1}{\sqrt{2\pi\sigma}}\exp\left\{-\frac{\kappa^2}{2\sigma}\right\}, \enskip \sigma>0,\enskip \kappa\in\mathbb{R}
\end{eqnarray}
also known as the source solution of the heat equation [see for instance Rosenbloom and Widder (1959)],
the transition probability of the process $\tilde{X}$ is:
\begin{eqnarray}
\nonumber G(t,x;\tau,y)&:=&\mathbb{P}^{t,x}(\tilde{X}_\tau\in dy)\\
\label{dd}&=&\frac{y}{x}\frac{s-t}{s-\tau}\frac{k(s-\tau,y)}{k(s-t,x)}\left[k\left(\tau-t,y-x\right)-k\left(\tau-t,x+y\right)\right]
\end{eqnarray}
for $0\leq t<\tau\leq s$, and $x,y\geq 0$.
\end{definition}
\subsection{Verification of the 3-dimensional Bessel bridge ({\it sketch\/})}
In order to verify that (\ref{sab}) and (\ref{bessel}) are equivalent, we will make use of ideas  described in for instance: Imhof (1984) or Chapter XI in  Revuz and Yor (2005). First, one derives the finite-dimen\-sional distribution of the conditional process $(Y_t,0\leq t\leq s|T_a=s)$ with $Y_t=a-\tilde{B}_t$. [A similar analysis, is carried out for the {\it Brownian bridge\/} in Problems 5.6.11 and 5.6.13, Karatzas and Shreve (1991).] Later, one may verify that $Y$ is a version of the process $\tilde{X}$, introduced in Definition \ref{bb3}.

To this end, let us first recall the following\\[0.3cm]
\noindent{\bf Remark.} ({\it Brownian motion absorbed at $y$}\/).
Consider the transition density of Brownian motion
\begin{eqnarray*}
G(s,x;t,y):=k(t-s,y-x)
\end{eqnarray*}
where $k$ is as in (\ref{ph}) and $y<a$. It follows---from the reflection principle---that the distribution of a Brownian motion, started at $x$ at time $s$ and absorbed at level $a$ is given by:
\begin{eqnarray*}
\mathbb{P}^{s,x}\left(B_t<y,T_a>t\right)&=&\mathbb{P}^{s,x}(B_t<y)-\mathbb{P}^{s,x}(B_t<y,T_a<t)\\
&=&\mathbb{P}^{s,x}(B_t<y)-\mathbb{P}^{s,x}(B_t>2a-y)\\
&=&\mathbb{P}(B_t-B_s<y-x)\\
&&\quad-\mathbb{P}(B_t-B_s>2a-y-x).
\end{eqnarray*}
Or, equivalently:
\begin{eqnarray}\label{absorbed}
\nonumber\mathbb{P}^{s,x}\left(B_t\in dy,T_a>t\right)&=&G(s,x;t,y)-G(s,x;t,2a-y)\\
&=&k(t-s,y-x)-k(t-x,2a-y-x).
\end{eqnarray}

\vglue 0.1cm
\noindent {\it Constructing the finite dimensional distribution of $Y$\/}. Next, given $0=t_0<t_1<\cdots<t_n<t_{n+1}=s$, and a positive sequence of values $x_1,\dots,x_n$ where $x_0=a$ and $x_{n+1}=0$, compute:
\begin{eqnarray*}
&&\tilde{\mathbb{P}}\left(Y_{t_1}\in dx_1,\dots,Y_{t_n}\in dx_n, T_a\in ds\right)\\
&&\quad\enskip=\tilde{\mathbb{P}}\left(Y_{t_i}\in dx_i,1\leq i\leq n, T_a\in ds\right)\\
&&\quad\enskip=\tilde{\mathbb{P}}\left(\tilde{B}_{t_i}\in d(a-x_i),1\leq i\leq n, T_a\in ds\right)\\
&&\quad\enskip=\tilde{\mathbb{P}}\left(\tilde{B}_{t_i}-\tilde{B}_{t_{i-1}}\in d(x_{i-1}-x_i),1\leq i\leq n, T_a\in ds\right)
\end{eqnarray*}
where the second equality follows from $Y_t=a-\tilde{B_t}$ and the third from the independence of increments.
In turn, this implies that
\begin{eqnarray*}
&&\tilde{\mathbb{P}}\left(Y_{t_1}\in dx_1,\dots,Y_{t_n}\in dx_n|T_a=s\right)\\
&&\enskip =\frac{\tilde{\mathbb{P}}\left(Y_{t_1}\in dx_1,\dots,Y_{t_n}\in dx_n, T_a\in ds\right)}{\mathbb{P}(T_a\in ds)}\\
&&\enskip=\frac{\mathbb{P}\left(T_{x_n}\in d(s-t_n)\right)}{\mathbb{P}\left(T_a\in ds\right)}\prod\limits_{j=1}^n\mathbb{P}^{t_{j-1},a-x_{j-1}}\Big{(}B_{t_j}\in d(a-x_j),T_a>t_j\Big{)}\\
&&\enskip=\frac{\mathbb{P}\left(T_{x_n}\in d(s-t_n)\right)}{\mathbb{P}\left(T_a\in ds\right)}\prod\limits_{j=1}^n\Big{[}G\left(t_{j-1},a-x_{j-1};t_j,a-x_j\right)\\
&&\enskip\qquad\qquad\qquad\qquad\qquad-G(t_{j-1},a-x_{j-1};t_j,a+x_j)\Big{]}
\end{eqnarray*}
the last equality follows from equation (\ref{absorbed})  (absorbed Brownian motion at $y$),
\begin{eqnarray*}
&&\enskip=\frac{\mathbb{P}\left(T_{x_n}\in d(s-t_n)\right)}{\mathbb{P}\left(T_a\in ds\right)}\prod\limits_{j=1}^n\Big{[}k(t_j-t_{j-1},x_j-x_{j-1})\\
&&\enskip\qquad\qquad\qquad\qquad\qquad-k(t_j-t_{j-1},x_j+x_{j-1})\Big{]}\\
&&\enskip=\frac{x_n}{a}\frac{s}{s-t_n}\frac{k(s-t_n,x_n)}{k(s,a)}\prod\limits_{j=1}^n\Big{[}k(t_j-t_{j-1},x_j-x_{j-1})\\
&&\enskip\qquad\qquad\qquad\qquad\qquad-k(t_j-t_{j-1},x_j+x_{j-1})\Big{]}\\
&&\enskip=\prod\limits_{j=1}^n\frac{x_i}{x_{i-1}}\frac{s-t_{i-1}}{s-t_i}\frac{k(s-t_i,x_i)}{k(s-t_{i-1},x_{i-1})}\\
&&\enskip\qquad\qquad\qquad\times\left[k(t_j-t_{j-1},x_j-x_{j-1})-k(t_j-t_{j-1},x_j+x_{j-1})\right].
\end{eqnarray*}
By inspection one observes that the transition probability of the conditioned process $Y$ is of the form (\ref{dd}). For the  remaining technical details the reader may consult Chapter XI at Revuz and Yor (2005).
\subsection{Hitting problems and 3-dimensional Bessel bridges}
From Proposition \ref{prop1} and Definition \ref{bb3}, we are now able to represent a class of first-hitting time problems in terms of a 3-dimensional Bessel process, namely:

\begin{theorem} Given that $h$ is as in (\ref{level}),
the process $\tilde{X}$ is a 3-dimensional Bessel bridge, Definition \ref{bb3}.
Then, the distribution of $T$ equals
\begin{eqnarray}
\nonumber\mathbb{P}(T<t)&=&
\int_0^t\tilde{\mathbb{E}}\left[\exp\left\{-\int_0^sf''(u)\tilde{X}_udu\right\}\right]\\
\label{cota}&&\quad\times \exp\left\{-\frac{1}{2}\int_0^s(f'(u))^2du-f'(0)a\right\}h(s,a)ds
\end{eqnarray}
\end{theorem}
\begin{proof}
Directly verified from Proposition \ref{prop1} and Definition \ref{bb3}.
\end{proof}
\section{Feynman-Kac and 3-dimensional Bessel bridge}\label{sec4}
In this section, we will provide a  Feynman-Kac representation [see for instance: Theorem 4.4.2 in Karatzas and Shreve (1991)] of 
\begin{eqnarray*}
v(t,a)=\tilde{\mathbb{E}}\left[\exp\left\{-\int_t^sf''(u)\tilde{X}_udu\right\}\Big{|}X_t=a\right],
\end{eqnarray*}
where $\tilde{X}$ is the 3-dimensional Bessel bridge, equation (\ref{bessel}), introduced in Definition \ref{bb3}.
The idea is to apply Ito's rule to the process $$v(y,X_y)\exp\left\{-\int^t_yf''(u)\tilde{X}_udu\right\},\enskip y\in[t,s],$$
given the following growth and/or boundary condition
\begin{eqnarray*}
0\leq v(t,a)\leq 1,\qquad 0\leq t<s,\enskip a\geq 0.
\end{eqnarray*}
We will also make use of  the fact that $f''(u)\tilde{X}_u\geq 0$ for $0\leq u<s$; that the process $\tilde{X}$ is continuous and strictly positive on $u\in[0,s)$, and that the moments of the running maximum of the process $\tilde{X}$ are known and finite.
\begin{theorem}\label{th1}
Suppose that $v(t,a):[0,s]\times \mathbb{R}^+\to\mathbb{R}^+$ is continuous,  of class $\mathbb{C}^{1,2}([0,s)\times\mathbb{R}^+),$and satisfies the Cauchy problem
\begin{eqnarray}
\label{cau}-\frac{\partial v}{\partial t}+f''(t)av&=&\frac{1}{2}\frac{\partial^2 v}{\partial a^2}+\left(\frac{1}{a}-\frac{a}{s-t}\right)\frac{\partial v}{\partial a}\qquad[0,s)\times\mathbb{R}^+,\\
\nonumber v(s,a)&=&1,\qquad a\in\mathbb{R}^+,
\end{eqnarray}
as well as
\begin{eqnarray*}
0\leq v(t,a)\leq 1\qquad \forall\enskip t,a\in\mathbb{R}^+.
\end{eqnarray*}
Then $v(t,a)$ admits the stochastic representation:
\begin{eqnarray*}
v(t,a)=\mathbb{E}^{t,a}\left[\exp\left\{-\int_t^sf''(u)\tilde{X}_udu\right\}\right].
\end{eqnarray*}
\end{theorem}

\begin{proof} We proceed as in the proof of Theorem 5.7.6, pp. 366--367, Karatzas \& Shreve (1991).
Applying Ito's rule to the process 
$$v(y,X_y)\exp\left\{-\int^t_yf''(u)\tilde{X}_udu\right\},$$  $y\in[t,s]$, we obtain, with $\tau_n:=\inf\{t\leq y\leq s| \tilde{X}_y\geq n\}$,
\begin{eqnarray*}
v(t,a)&=&\tilde{\mathbb{E}}^{t,a}\left[\exp\left\{-\int_t^sf''(u)\tilde{X}_udu\right\}\mathbb{I}_{(\tau_n>s)}\right]\\
&&+\tilde{\mathbb{E}}^{t,a}\left[v(\tau_n,X_{\tau_n})\exp\left\{-\int_t^{\tau_n}f''(u)\tilde{X}_udu\right\}\mathbb{I}_{(\tau_n\leq s)}\right]
\end{eqnarray*}
The second term converges to zero as $n\to\infty$, since
\begin{eqnarray*}
&&\tilde{\mathbb{E}}^{t,a}\left[v(\tau_n,X_{\tau_n})\exp\left\{-\int_t^{\tau_n}f''(u)\tilde{X}_udu\right\}\mathbb{I}_{(\tau_n\leq s)}\right]\\&&\qquad\qquad\leq \tilde{\mathbb{E}}^{t,x}\left[v(\tau_n,\tilde{X}_{\tau_n})\mathbb{I}_{(\tau_n\leq s)}\right]\\
&&\qquad\qquad\leq\tilde{\mathbb{P}}^{t,x}(\tau_n\leq s)\\
&&\qquad\qquad=\tilde{\mathbb{P}}^{t,x}\left(\max\limits_{t\leq\theta\leq s}\tilde{X}_\theta\geq n\right)\\
&&\qquad\qquad\leq \frac{\tilde{\mathbb{E}}^{t,x}\left[\max\limits_{t\leq\theta\leq s}\tilde{X}_\theta^{2m}\right]}{n^{2m}}
\end{eqnarray*}
[see Pitman and Yor (1998) for the moments of the running maximum of $\tilde{X}$]. Finally, the first term converges to
\begin{eqnarray*}
\tilde{\mathbb{E}}^{t,a}\left[\exp\left\{-\int_t^sf''(u)\tilde{X}_udu\right\}\right]
\end{eqnarray*}
either by the dominated or by the monotone convergence theorem.
\end{proof}

Our next step is to relate Theorem \ref{th1} to two alternative, better known problems. Namely, that of Schr\"odinger's equation with linear potential.  To this end, we first have:
\begin{proposition} Solutions to (\ref{cau}) may be of the following form:
\begin{eqnarray}\label{par}
v(t,a)=\frac{w(t,a)}{h(s-t,a)}
\end{eqnarray}
where
\begin{eqnarray}\label{schro}
-w_t(t,a)+f''(t)aw(t,a)=\frac{1}{2}w_{aa}(t,a),\quad [0,s)\times\mathbb{R}^+
\end{eqnarray}
and $h(s,a)$ is  as in (\ref{level}).
\end{proposition}
\begin{proof} 
Setting
$$u(t,a)=1/h(s-t,a)$$
and 
\begin{eqnarray}\label{vx}
v(t,a):=u(t,a)w(t,a),
\end{eqnarray} we have that
\begin{eqnarray*}
u_t(t,a)&=&\left[\frac{a^2}{2(s-t)^2}-\frac{3}{2(s-t)}\right]u(t,a)\\
 u_a(t,a)&=&-\left[\frac{1}{a}-\frac{a}{s-t}\right]u(t,a)\\
u_{aa}(t,a)&=&\left[\frac{2}{a^2}+\frac{a^2}{(s-t)^2}-\frac{1}{s-t}\right]u(t,a).
\end{eqnarray*}
Alternatively, from (\ref{vx}) it follows
\begin{eqnarray*}
v_t&=&u_tw+uw_t\qquad v_a=u_aw+uw_a\\
v_{aa}&=&u_{aa}w+2u_xw_w+uw_{aa}.
\end{eqnarray*}
Hence from (\ref{cau}) and (\ref{schro}) we conclude that
\begin{eqnarray*}
\left[-u_t-\frac{1}{2}u_{aa}-\left(\frac{1}{a}-\frac{a}{s-t}\right)u_a\right]w=\left[u_a+\left(\frac{1}{a}-\frac{a}{s-t}\right)u\right]w_a,
\end{eqnarray*}
as claimed.
\end{proof}

The previous Proposition may be generalized by first noting that
\begin{eqnarray*}
\left(\frac{1}{x}-\frac{x}{s-t}\right)=\frac{h_x(s-t,x)}{h(s-t,x)}
\end{eqnarray*}
for $h$ as in (\ref{level}). Second, recall that $h$ is a solution to the heat equation
\begin{eqnarray*}
h_t=\frac{1}{2}h_{xx}.
\end{eqnarray*}
This leads to
\begin{proposition}\label{gg}
 If $h$, $v$, and $w$ satisfy the following system of partial differential equations
\begin{eqnarray}\label{system}
\left\{
\begin{array}{l}
-\frac{1}{\sigma_t}\cdot h_t=\frac{1}{2}h_{xx}\\[0.2cm]
-v_\tau+k(\tau,x)v=\frac{1}{2}v_{xx}+\frac{h_x}{h}v_x\\[0.2cm]
-w_\tau+k(\tau,x)w=\frac{1}{2}w_{xx}
\end{array}
\right.
\end{eqnarray}
where $\sigma_t$ is only a function of $t$, 
then they satisfy the identity:
\begin{eqnarray}\label{ratio}
v(\tau,x)=\frac{w(\tau,x)}{h(\tau,x)},
\end{eqnarray}
where
\begin{eqnarray}\label{time}
\tau:=\int_{s}^t\sigma_udu.
\end{eqnarray}
\end{proposition}
\begin{proof}
Given $\tau$ is as in (\ref{time}):
\begin{eqnarray*}
\frac{\partial h}{\partial t}=\frac{\partial \tau}{\partial t}\frac{\partial h}{\partial \tau}.
\end{eqnarray*}
Next, we use (\ref{ratio}) to compute
\begin{eqnarray*}
v_\tau=\frac{w_\tau h-w h_\tau}{h^2},\qquad v_x=\frac{w_xh-wh_x}{h^2}\\
v_{xx}=\frac{(w_{xx}h-wh_{xx})h^2-2hh_x(w_xh-wh_x)}{h^4}.
\end{eqnarray*}
Finally, substituting in the second equation in (\ref{system}) we have:
\begin{eqnarray*}
-\frac{w_\tau h-wh_\tau}{h^2}+k\frac{w}{h}&=&\frac{1}{2}\frac{w_{xx}}{h}-\frac{1}{2}\frac{h_{xx}w}{h^2}\\
&&-\frac{h_x(w_xh-wh_x)}{h^3}+\frac{h_x}{h}\left[\frac{w_xh-wh_x}{h^2}\right]
\end{eqnarray*}
which implies
\begin{eqnarray*}
-\frac{w_\tau}{h}+k\frac{w}{h}=\frac{1}{2}\frac{w_{xx}}{h}
\end{eqnarray*}
as claimed.
\end{proof}

\section{Schr\"odinger's equation with linear potential}\label{sec5}

Equations (\ref{par}) and (\ref{schro}), together with the growth condition
\begin{eqnarray}\label{bo}
0\leq v(t,a)\leq 1,\enskip 0\leq t\leq s,\enskip a\in\mathbb{R}^+
\end{eqnarray}
will allow us to introduce a couple of Schr\"odinger-type problems.\\[0.2cm]
\noindent{\it First Schr\"odinger problem\/}. From (\ref{bo}) and recalling the definition of $h$ in (\ref{level}):
$$h(t,x)=\frac{|x|}{\sqrt{2\pi t^3}}\exp\left\{-\frac{x^2}{2t}\right\},$$ 
we have that
\begin{eqnarray*}
0\leq w(t,a)\leq h(s-t,a),\quad \forall\,0\leq t\leq s,\enskip a\in\mathbb{R}^+.
\end{eqnarray*}

Furthermore, note that
\begin{eqnarray*}
\lim\limits_{t\to s}h(s-t,a)&=&0\\
\lim\limits_{a\to 0}h(s-t,a)&=&\delta(s).
\end{eqnarray*}
The last statement should be interpreted in the following sense: Suppose $\lambda:\mathbb{R}^+\to\mathbb{R}^+$ is a $\mathbb{C}^1$ function, then
\begin{eqnarray*}
\lim\limits_{x\to 0}\int_0^sh(s-t,x)\lambda(t)dt=\lambda(s),\qquad 0\leq t<s<\infty.
\end{eqnarray*}

This leads to  our first representation:

\begin{remark}\label{rem1} Let $v$ be as in Theorem \ref{th1} and $h$ as in  (\ref{level}). Then the following relationship holds
$$v(t,a)=\frac{w(t,a)}{h(s-t,a)},\qquad t\in[0,s]\times\mathbb{R}^+.$$
Furthermore $w$ is a solution to the Schr\"odinger equation with linear potential
\begin{eqnarray}\label{pde1}
&&-\frac{\partial w}{\partial t}(t,a)+f''(t)aw(t,a)\\
\nonumber&&\qquad\qquad=\frac{1}{2}\frac{\partial ^2w}{\partial a^2}(t,a),\quad t\in [0,s],\enskip a\in \mathbb{R}^+
\end{eqnarray}
satisfying the following boundary conditions
\begin{eqnarray*}
\lim\limits_{t\to s}w(t,a)&=&0\\
\lim\limits_{a\to 0}w(t,a)&=&\delta(s),
\end{eqnarray*}
as well as the compatibility condition
\begin{eqnarray*}
\lim\limits_{t\to s}\frac{w(t,a)}{h(s-t,a)}=1.
\end{eqnarray*}
\end{remark}

For our second problem, we will make use of the following:
\begin{proposition}\label{ggt}
If $h$, $k$, $v$, and $w$ satisfy the following system of partial differential equations
\begin{eqnarray*}
\left\{
\begin{array}{l}
-\frac{1}{\sigma_t}\cdot h_t=\frac{1}{2}h_{xx}\\[0.2cm]
-k_\tau=\frac{1}{2}k_{xx}\\[0.2cm]
-v_\tau+\sigma(\tau,x)v=\frac{1}{2}v_{xx}+\frac{k_x}{k}v_x\\[0.2cm]
-w_\tau+\left[\sigma(\tau,x)+\left(\frac{k_x}{k}-\frac{h_x}{h}\right)\frac{h_x}{h}\right]w=\frac{1}{2}w_{xx}+w_x\left(\frac{k_x}{k}-\frac{h_x}{k}\right)
\end{array}
\right.
\end{eqnarray*}
where $\sigma_t$ is only a function of $t$, then they satisfy the identity:
\begin{eqnarray*}
v(\tau,x)=\frac{w(\tau,x)}{h(\tau,x)},
\end{eqnarray*}
where
\begin{eqnarray}
\tau:=\int_{s}^t\sigma_udu.
\end{eqnarray}
\end{proposition}
\begin{proof} We proceed as in  the proof of Proposition \ref{gg}
\end{proof}

\noindent{\it Second Schr\"odinger problem\/}. Given the fundamental or source solution of the heat equation $k$, equation (\ref{ph}), we may relate $v$ to another Schr\"odinger-type problem, by setting
\begin{eqnarray*}
v(t,a)=\frac{u(t,a)}{k(s-t,a)}
\end{eqnarray*}
and making use of Proposition \ref{ggt}:
\begin{remark}\label{rem2} Let $v$ be as in Theorem \ref{th1} and $k$ as in equation (\ref{ph}), then the following relationship holds:
$$v(t,a)=\frac{u(t,a)}{k(s-t,a)},\qquad t\in[0,s]\times\mathbb{R}^+.$$
Furthermore, it follows from Proposition \ref{ggt}, that $u$ is a solution to the Schr\"odinger equation with linear potential:
\begin{eqnarray}\label{pde2}
&&-\frac{\partial u}{\partial t}(t,a)+\left[-\frac{1}{(s-t)}+f''(t)a\right]u(t,a)\\
\nonumber&&\qquad\quad=\frac{1}{2}\frac{\partial ^2u}{\partial a^2}(t,a)+\frac{1}{a}\frac{\partial u}{\partial a}(t,a),\quad t\in [0,s],\enskip a\in \mathbb{R}^+
\end{eqnarray}
satisfying the following boundary conditions:
\begin{eqnarray*}
\lim\limits_{t\to s}u(t,a)&=&\delta(a)\\
\lim\limits_{a\to 0}u(t,a)&=&0,
\end{eqnarray*}
as well as the compatibility condition
\begin{eqnarray*}
\lim\limits_{t\to s}\frac{u(t,a)}{k(s-t,a)}=1.
\end{eqnarray*}
\end{remark}

The boundary conditions of this second problem follow, once more, by using the growth condition in Theorem \ref{th1}, i.e.
\begin{eqnarray*}
0\leq u(t,a)\leq k(s-t,a)\qquad \forall\, t\in[0,s],\enskip a\in\mathbb{R}.
\end{eqnarray*}

We conclude this section by making a few remarks. First, note that the problem described in Remark \ref{rem2} is a {\it Cauchy\/} problem, the one appearing in Remark \ref{rem1} is not. 

We have introduced the consistency conditions in Remarks \ref{rem1} and \ref{rem2} in order to make them compatible with Theorem \ref{th1}. Boundary conditions for the density of $T$ will be derived in the following section.

The next  observation is that for the equations (\ref{pde1}) and (\ref{pde2}) one may construct {\it particular\/} solutions. There is in fact {\it one\/} solution of (\ref{pde1}) which may be used to find a new set of solutions of (\ref{pde2}). To do so, one may use a technique introduced in Bluman and Shtelen (1996). However,  solutions to the problems described in Remarks  \ref{rem1} and \ref{rem2}, for arbitrary bounds $f$---which do not involve {\it integral equations\/}---have yet to be found. 

This last remark will be clarified in the following section. The idea stems from the following fact (yet to be shown): There is a particular set of solutions of either (\ref{pde1}) and/or (\ref{pde2}) which transform  our original problem---that is, finding the density of hitting times---into that of finding particular solutions to the heat equation killed at a moving boundary! Unfortunately, in practice this is {\it not\/} very useful to us, since the problem leads back to Volterra integral equations as described in De Lillo and Fokas (2007).

A final curious remark is that equation (\ref{pde1}) may be transformed into a non-homogeneous Burgers' equation. The  homogeneity is only time dependent and it corresponds precisely to $f''$:
\begin{eqnarray}\label{burgers}
\kappa_t(t,a)+\kappa(t,a)\cdot \kappa_a(t,a)-\frac{1}{2}\kappa_{aa}(t,a)=f''(t).
\end{eqnarray}
Alternatively, $\kappa$ and  $w$---equation (\ref{pde1})---introduced in Remark \ref{rem1},  are related through:
\begin{eqnarray*}
\kappa(t,a)=-\frac{\partial}{\partial a}\log w(t,a).
\end{eqnarray*}
From Proposition \ref{gg} it follows:
\begin{proposition}
If $h$, $v$, $w$ and $u$ satisfy the following system of partial differential equations
\begin{eqnarray*}
\left\{
\begin{array}{l}
-\frac{1}{\sigma_t}\cdot h_t=\frac{1}{2}h_{xx}\\[0.2cm]
-v_\tau+k(\tau,x)v=\frac{1}{2}v_{xx}+\frac{h_x}{h}v_x\\[0.2cm]
-w_\tau+k(\tau,x)w=\frac{1}{2}w_{xx}\\[0.2cm]
-u_\tau+u\cdot u_x=\frac{1}{2}u_{xx}+k_x
\end{array}
\right.
\end{eqnarray*}
where $\sigma_t$ is only a function of $t$, then they satisfy the identities:
\begin{eqnarray*}
v(\tau,x)=\frac{w(\tau,x)}{h(\tau,x)},\qquad u=-\frac{w_x}{w},\qquad w=e^{-\int udx},
\end{eqnarray*}
where
\begin{eqnarray*}
\tau:=\int_{s}^t\sigma_udu.
\end{eqnarray*}
\end{proposition}
Of course, at first sight, this is  not so surprising since this relationship is well known. The possibly ``interesting'' part lies in the fact that so many probability--related papers have been written regarding Burgers' equation (\ref{burgers}), see for instance Barndorff-Nielsen and Leonenko (2005) and references therein.
\section{Particular solutions to Schr\"odinger's equation with linear potential (heat equation)}\label{sec6}
Particular solutions to Schr\"odinger's equation with linear potential have been obtained, for instance, by Feng (2001) an references therein.

The strategy employed in Feng (2001) is not only interesting but it is also useful. This follows since it ultimately leads to heat equation problems. In this section, we propose yet another {\it equivalent\/} approach using Fourier transforms:
\begin{theorem}
Solutions to (\ref{schro}) are given by
\begin{eqnarray}
\label{sch1}&&\\
\nonumber w(t,a)&=&e^{\frac{1}{2}\int_t^s(f'(u))^2du+af'(t)}\frac{1}{2\pi}\int_{-\infty}^\infty \Pi(y)e^{-\frac{1}{2}y^2(s-t)+iy(a+\int_t^sf'(u)du)}dy\\
\nonumber&=&e^{\frac{1}{2}\int_t^s(f'(u))^2du+af'(t)}\omega\left(s-t,a+\int_t^sf'(u)du\right)
\end{eqnarray}
where $\omega$ is a solution to the heat equation and $\Pi$ is an arbitrary function, as long as the integral is well defined.
\end{theorem}

\begin{proof} Let
\begin{eqnarray}\label{fou}
\hat{w}(t,\lambda):=\int\limits_{-\infty}^\infty e^{-i\lambda a}w(t,a)da.
\end{eqnarray}
Applying the Fourier transform to (\ref{schro}) we have
\begin{eqnarray*}
-\hat{w}_t(t,\lambda)+if''(t)\hat{w}_\lambda(t,\lambda)+\frac{1}{2}\lambda^2\hat{w}(t,\lambda)=0\qquad i:=\sqrt{-1}.
\end{eqnarray*}
Next, set $y=\lambda+if'(t)$ and $\hat{w}(t,\lambda)=\tilde{w}(t,y)$, that is:
\begin{eqnarray*}
\hat{w}_t=\tilde{w}_t+if''(t)\tilde{w}_y,\qquad \hat{w}_\lambda=\tilde{w}_y
\end{eqnarray*}
which after substitution in (\ref{fou}) leads to
\begin{eqnarray}\label{wi}
-\tilde{w}_t(t,y)+\frac{1}{2}(y-if'(t))^2\tilde{w}(t,y)=0.
\end{eqnarray}
Consequently, setting $\Pi(y)=\tilde{w}(0,y)$  it follows that
\begin{eqnarray}\label{no}
\nonumber\tilde{w}(t,y)&=&\Pi(y)\exp\left\{-\frac{1}{2}\int_t^s(y-if'(u))^2du\right\}\\
&=&\Pi(y)\exp\left\{-\frac{1}{2}y^2(s-t)+iy\int_t^sf'(u)du+\frac{1}{2}\int_t^s(f'(u))^2du\right\}.
\end{eqnarray}
[Note that equation (\ref{no}) will be a solution of (\ref{wi}) for arbitrary $\Pi$.] Alternatively, this implies that
\begin{eqnarray*}
w(t,a)&=&\frac{1}{2\pi}\int_{-\infty}^\infty \Pi(y)e^{-\frac{1}{2}y^2(s-t)+iy\int_t^sf'(u)du+\frac{1}{2}\int_t^s(f'(u))^2du}e^{i\lambda a}d\lambda\\
&=&\frac{1}{2\pi}\int_{-\infty}^\infty \Pi(y)e^{-\frac{1}{2}y^2(s-t)+iy\int_t^sf'(u)du+\frac{1}{2}\int_t^s(f'(u))^2du}e^{iy a+af'(t)}dy\\
&=&e^{\frac{1}{2}\int_t^s(f'(u))^2du+af'(t)}\frac{1}{2\pi}\int_{-\infty}^\infty \Pi(y)e^{-\frac{1}{2}y^2(s-t)+iy(a+\int_t^sf'(u)du)}dy
\end{eqnarray*}
as claimed.
\end{proof}
This leads to our already anticipated claim, that boundary crossing probabilities, are equivalent to a set of heat equation problems.
\begin{theorem}\label{th3} The density $\varphi$ of the first passage time $T$ defined in (\ref{stop}) is related to a solution of the heat equation $\omega$ as follows
$$\varphi(s,a)=\omega\left(s,a+\int_0^sf'(u)du\right),$$
and it is bounded by
\begin{eqnarray*}
&&h(s,a)e^{-af'(0)-\frac{1}{2}\int_0^s(f'(u))^2du-\int_0^sf''(u)\tilde{\mathbb{E}}^a(\tilde{X}_u)du}\\
&&\qquad \qquad\leq \omega\left(s, a+\int_0^sf'(u)du\right)\\
&&\qquad\qquad \qquad\qquad\leq h(s,a)e^{-af'(0)-\frac{1}{2}\int_0^s(f'(u))^2du}
\end{eqnarray*}
where $h$ is as in (\ref{level}).
\end{theorem}
\begin{proof}
It follows from  Jensen's inequality,
\begin{eqnarray*}
\exp\left\{-\int_0^sf''(u)\tilde{\mathbb{E}}^{t,a}(\tilde{X}_u)du\right\}\leq\tilde{\mathbb{E}}^{t,a}\left[\exp\left\{-\int_0^sf''(u)\tilde{X}_udu\right\}\right]\leq 1,
\end{eqnarray*}
equation (\ref{schro}), and (\ref{sch1}).
\end{proof}
As a corollary we obtain the {\it necessary\/} boundary conditions for density $\varphi$:\\[0.3cm]
\begin{corollary}\label{cor1} From the previous Theorem, the following boundary conditions hold: 
\begin{eqnarray}\label{equal}
\lim\limits_{a\to 0} \omega\left(s, a+\int_0^sf'(u)du\right)&=&\delta(0),\\
\nonumber\lim\limits_{s\to 0}\omega\left(s,a+\int_0^sf'(u)du\right)&=&0.
\end{eqnarray}
and
\begin{eqnarray*}
&&\frac{1}{\sqrt{2\pi s^3}}e^{-\frac{1}{2}\int_0^s(f'(u))^2du-2\sqrt{\frac{2}{\pi}}\int_0^sf''(u)\sqrt{s-u}\sqrt{\frac{u}{s}}du}\\
&&\qquad\qquad\quad\leq\lim\limits_{a\to 0}\omega_a\left(s,a+\int_0^sf'(u)du\right)\\
&&\qquad\qquad\qquad\qquad\quad\quad\leq \frac{1}{\sqrt{2\pi s^3}}e^{-\frac{1}{2}\int_0^s(f'(u))^2du}
\end{eqnarray*}
or given the definition of {\it Fractional Integral}
$$
J^\alpha g(x):=\frac{1}{\Gamma(\alpha)}\int_0^x(x-y)^{\alpha-1}g(y)dy,\quad x>0,\quad \alpha\in \mathbb{R}^+
$$
we have
\begin{eqnarray*}
&&\frac{1}{\sqrt{2\pi s^3}}e^{-\frac{1}{2}\int_0^s(f'(u))^2du-s^{-1/2}J^{3/2} f''(s)\sqrt{2s}}\\
&&\qquad\qquad\leq\lim\limits_{a\to 0}\omega_a\left(s,a+\int_0^sf'(u)du\right)\\
&&\qquad\qquad\qquad\qquad\leq \frac{1}{\sqrt{2\pi s^3}}e^{-\frac{1}{2}\int_0^s(f'(u))^2du}.
\end{eqnarray*}
\end{corollary}
That is, Theorem \ref{th3} together with Corollary \ref{cor1} uniquely characterize the density of $T$. Alternatively, from Theorem \ref{th3}, we have obtained---for free---lower bounds for the problems described in Remarks \ref{rem1} and \ref{rem2}.  
\section{Concluding Remarks}\label{sec7}

As was already mentioned in the Introduction, the aim of this paper is not to provide a new set of solutions to the problem of boundary crossing. Our goal is instead conceptual, in that we first make use of  Girsanov's theorem to  connect hitting time problems to Bessel bridges. We do so by use of the Optional Sampling Theorem. Next, we relate exponential functionals of 3 dimensional Bessel bridges to  Schr\"odinger's equation with linear potential. This is done by {\it uniquely\/} relating both problems through a Feynman-Kac representation. We conclude by characterizing the density of $T$ in terms of a heat-equation problem and/or  two alternate Schr\"odinger-equation problems. 

Girsanov's link is useful since it suggests how to extend this idea to higher dimensions or more general diffusions. Alternatively, the construction of Bessel bridges points to the fact that the concept can be extended. From the derivation of our Feynman-Kac representation we have learned that in some particular cases, sufficient conditions can be relaxed. Finally, Schr\"odinger's equation with linear potential may be the way in which boundary crossing probabilities of some ``higher order'' processes  can be described.  For example, instead of having a parabolic equation---which describes the transition probability of Brownian motion---as in (\ref{pde1}), it could be that a process with transition probability
\begin{eqnarray*}
\frac{\partial\theta}{\partial t}(t,x)=\alpha\frac{\partial^n\theta}{\partial x^n}(t,x),\quad[0,\infty)\times \mathbb{R}, \enskip n\geq 2,
\end{eqnarray*}
has a first passage time probability which might be described by a Schr\"odinger type-problem:
\begin{eqnarray*}
\frac{\partial w}{\partial t}(t,x)+xf''(t)w(t,x)=\alpha\frac{\partial^nw}{\partial x^n}(t,x),\quad [0,\infty)\times \mathbb{R}.
\end{eqnarray*}
In particular note that the techniques used in the proof of Theorem \ref{th3} still hold.

\end{document}